\documentclass[11pt]{article}

\usepackage{amsmath}
\usepackage{amsfonts}
\usepackage{amssymb}
\usepackage{amsthm}
\usepackage{graphics}
\usepackage{amscd}
\usepackage{graphicx,epsfig}
\usepackage{color}
\usepackage[all]{xy}
\usepackage{url}
\usepackage{float}
\usepackage[caption = false]{subfig}

\renewcommand{\epsilon}{\varepsilon}

\newcommand{\boH}{\mathcal{H}}

\newcommand{\boN}{\mathcal{N}}

\newcommand{\R}{\mathbb{R}}

\newcommand{\Z}{\mathbb{Z}}
\newcommand{\C}{\mathbb{C}}

\renewcommand{\H}{\mathbb{H}}
\renewcommand{\S}{\mathbb{S}}
\newcommand{\N}{\mathbb{N}}

\newcommand{\eps}{\varepsilon}

\newcommand{\dd}{\mathrm{d}}

\newcommand{\Ome}{\Omega}

\newtheorem{thm}{Theorem}
\newtheorem{cor}[thm]{Corollary}
\newtheorem{prop}[thm]{Proposition}
\newcommand{\inter}[1]{\overset{\circ}{#1}}
\newcommand{\barre}[1]{\overline{#1}}
\renewcommand{\phi}{\varphi}

\newtheorem*{thm*}{Theorem}

\newtheorem{claim}{Claim}
\newtheorem*{claim*}{Claim}

\theoremstyle{remark}

\newtheorem{remarq}{Remark}
\newtheorem*{rem*}{Remark}

\newcounter{remark}

\newcounter{case}

\newcounter{construction}

\newcounter{fact}

\newcounter{step}

\DeclareMathOperator{\Ric}{Ric}

\theoremstyle{plain}
\newtheorem*{thA}{Theorem A}
\newtheorem*{thB}{Theorem B}
\newtheorem*{thC}{Theorem C}
\newtheorem*{thD}{Theorem D}

\title{Minimal surfaces in finite volume non compact hyperbolic $3$-manifolds}
\author{Pascal Collin, Laurent Hauswirth, Laurent Mazet, Harold Rosenberg
\thanks{The authors were partially supported by the ANR-11-IS01-0002 grant.}} 

\begin{document}
\maketitle

\begin{abstract}
We prove there exists a compact embedded minimal surface in a complete finite
volume hyperbolic $3$-manifold $\boN$. We also obtain a least area,
incompressible, properly embedded, finite topology, $2$-sided surface. We prove
a properly embedded minimal surface of bounded curvature has finite topology.
This determines its asymptotic behavior. Some rigidity theorems are obtained.
\end{abstract}

\section{Introduction}

There has been considerable progress on the study of properly embedded minimal
surfaces in euclidean $3$-space. We now know all such orientable surfaces that
are planar domains; they are planes, helicoids, catenoids and Riemanns'minimal
surfaces. Also we understand the geometry of properly embedded periodic minimal
surfaces, that are finite topology in the quotient.

In hyperbolic $3$-space, there is no classification of this nature. A continuous
rectifiable curve in (the boundary at infinity of $\H^3$) is the asymptotic
boundary of a least area embedded simply connected surface.

In this paper we study the existence of periodic minimal surfaces in $\H^3$. More
precisely, we consider surfaces in complete non compact hyperbolic $3$-manifolds
$\boN$ of finite volume. In the following of the paper, we will refer to such
manifolds $\boN$ as hyperbolic cusp manifolds. In a closed hyperbolic manifold
(or any closed Riemannian $3$-manifold), there is always a compact embedded
minimal surface \cite{Pit}. They cannot be of genus zero or one, but there are
many higher genus such surfaces. The existence and deformation theory of such
surfaces was initiated by K. Uhlenbeck \cite{Uhl}.

Hyperbolic cusp manifolds play an important
role in the theory of closed hyperbolic 3-manifolds. Many link complements in
the unit 3-sphere have such a finite volume hyperbolic structure. Given any $V >
0$, Jorgensen proved there are a finite number of such $\boN$ of volume $V$. Then
Thurston proved that a closed hyperbolic $3$-manifold of volume less than V can be
obtained from this finite number of manifolds $\boN$ given by Jorgensen, by
hyperbolic Dehn surgery on at least one of the cusp ends (see \cite{BePe} for
details).

We will prove there is a compact embedded minimal surface in any complete
hyperbolic $3$-manifold of finite volume. Since such a non compact manifold
$\boN$ is "not convex at infinity", minimization techniques do not produce such
a minimal surface. To understand this the reader can verify that on a complete
hyperbolic $3$-punctured $2$-sphere, there is no simple closed geodesic. In
dimension $3$, a min-max technique, together with several maximum principles in
the cusp ends of $\boN$, will produce compact embedded minimal surfaces.

We will give two existence results of embedded compact minimal
surfaces. 
\begin{thA}
There is a compact embedded minimal surface $\Sigma$ in $\boN$.
\end{thA}

\begin{thB}
Let $S$ be a closed orientable embedded surface in $\boN$ which is not a
$2$-sphere or a torus. If $S$ is incompressible and non-separating, then $S$ is
isotopic to a least area embedded minimal surface.
\end{thB}

Concerning properly embedded non compact minimal surfaces, there are already 
existence results due to Hass, Rubinstein and Wang \cite{HaRuWa} and Ruberman
\cite{Rub}. Using different arguments, we give an other proof of Ruberman's
minimization result.

\begin{thm*}
Let $S$ be a properly embedded, non compact, finite topology, incompressible,
non separating surface in $\boN$. Then $S$ is isotopic to a least area embedded
minimal surface.
\end{thm*}

The surfaces produced by the above theorem have bounded curvature. Actually the
techniques we develop enable us to prove:
\begin{thC}
Let $\Sigma$ be a properly embedded minimal surface in $\boN$ of bounded
curvature. Then $\Sigma$ has finite topology.
\end{thC}
Since stable minimal surfaces have bounded curvature we conclude:
\begin{cor}
A properly embedded stable minimal surface in $\boN$ has finite topology.
\end{cor}

Finite topology is particularly interesting here due to the Finite Total
curvature theorem below that describes the geometry of the ends of a properly immersed
minimal surface in $\boN$ of finite topology
\begin{thm}[Collin, Hauswirth, Rosenberg \cite{CoHaRo}] \label{th:ftc}
A properly immersed minimal surface $\Sigma$ in $\boN$ of finite topology has
finite total curvature
$$
\int_\Sigma K_\Sigma=2\pi\chi(\Sigma)
$$
Moreover, each end $A$ of $\Sigma$ is asymptotic to a totally geodesic $2$-cusp
end in an end $C$ of $\boN$.
\end{thm}
We will make precise these notions.

The simplest example of a surface $\Sigma$ with finite topology appearing in the
above Theorem is a $3$-punctured sphere. Actually, minimal $3$-punctured spheres
are totally
geodesic.
\begin{thD}
A proper minimal immersion of a $3$-punctured sphere in $\boN$ is totally
geodesic.
\end{thD}

The paper is organized as follows. In Section \ref{sec:discussion}, we make some
general remarks on the geometry of cusp manifolds stating some results of
Jorgensen, Thurston and Adams. In Section \ref{sec:sphere}, we consider
$3$-punctured spheres in hyperbolic cusp manifolds and prove Theorem D. In
Section \ref{sec:maxtrans}, we study minimal surfaces entering the ends of
hyperbolic cusp manifolds $\boN$. We prove two maximum principles which govern
the geometry of minimal surfaces in the ends of $\boN$. We also establish a
transversality result which is used to study annular ends of minimal surfaces. Section
\ref{sec:compact} proves Theorems A and B, the existence of compact embedded minimal
surfaces in hyperbolic cusp manifolds. Section \ref{sec:noncompact} proves the minimization result in the non compact case. Then in Section \ref{sec:examples}, we present several examples to illustrate these theorems.

\section{Some discussion of the manifolds $\boN$}
\label{sec:discussion}

In this section we recall some facts about the geometry of a non compact
hyperbolic $3$-manifold $\boN$ of finite volume.

Such $\boN$ are the union of a compact submanifold bounded by mean concave mean
curvature one tori, and a finite number of ends, each end isometric to a quotient of
a horoball of $\H^3$ by a $\Z^2$ group of parabolic isometries leaving the
horoball invariant. The horospheres in this horoball quotient to mean curvature
one tori in $\boN$.

An end of $\boN$ can be parametrized by $M=\{(x,y,z)\in\R^3\,|\, z\ge 1/2\}$
modulo a group $G=G(v_1,v_2)$, generated by two translations by linearly
independent horizontal vectors $v_1,v_2\in \R^2\times\{0\}$.

The end $C=M/G$ is endowed with the quotient of the hyperbolic metric of $M$ :
$$
g_\H=\frac1{z^2}(dx^2+dy^2+dz^2)=\frac1{z^2}dX^2.
$$

The horospheres $\{z=c\}$ quotient to tori $T(c)$, of mean curvature
one with respect to the unit normal vector $z\partial z$. The vertical curves
$\{(x,y)=\textrm{constant}\}$ are geodesics orthogonal to the tori $T(c)$, with
arc length given by $s=\ln z$. The induced metric on $T(c) $ is flat and lengths
on $T(c)$ decrease exponentially as $s\rightarrow \infty$.

We will denote by $T(a,b)$ the subset $\{a\le z\le b\}$ of $C$.

The Euclidean planes $\{ax+by=c\}$ are totally geodesic surfaces in $C$. When they
are properly embedded in $C$, they are the totally geodesic $2$-cusp ends in $C$
that appears in the Finite Total Curvature Theorem above.

Define $\Lambda(C)=\max\{|v_1|,|v_2|\}$ with $|v|$ the Euclidean norm. We notice
that we have made a choice of generators $v_1,v_2$ of the group $G$, so
the value $\Lambda(C)$ depends on this choice (we can minimize the value of
$\Lambda$ among all choices but it is not important in the following).

\begin{remarq}\label{rem:isom}
The above notations are well adapted to study the geometry close to $z=1$. For
$z_0$ larger than $1$, let $H$ be the map $(x,y,z)\mapsto (z_0x,z_0y,z_0z)$
which sends $M$ to $\R^2\times[z_0/2,+\infty)$. This map gives us then a chart of
$C'=\{z\ge z_0/2\}\subset C$ parametrized by $\{z'\ge 1/2\}$ with
$\Lambda(C')=\Lambda(C)/z_0$. So, considering a part of the end that is
sufficiently far away, we can always assume that $\Lambda(C)$ is small.
\end{remarq}

We mention two theorems concerning the manifolds $\boN$.
\begin{thm}[Jorgensen]
Given $V>0$, there exist a finite number of such manifolds $\boN$ whose volume
is equal to $V$.
\end{thm}
\begin{thm}[Thurston]
Any compact hyperbolic $3$-manifold $M^3$, $\partial M^3=\emptyset$, with
$Vol(M)<V$ is obtained from the finite number of $\boN$ given by
Jorgensens' theorem, by hyperbolic Dehn surgery on at least one of the cusp ends.
\end{thm}

Concerning surface theory in $\boN$, we mention one theorem that inspired
Theorem D.
\begin{thm}[Adams \cite{Ada}]
Let $\Sigma$ be a properly embedded $3$-punctured sphere in $\boN$, $\Sigma$
incompressible. Then $\Sigma$ is isotopic to a totally geodesic $3$-punctured
sphere in $\boN$.
\end{thm}


\section{Minimal $3$-punctured spheres are totally geodesic}
\label{sec:sphere}


In this section we prove that, under some hypotheses, a minimal surface is
totally geodesic. We first have the following result.

\begin{thD}
A proper minimal immersion of a $3$-punctured sphere in $\boN$ is totally
geodesic. Moreover, it is $\pi_1$ injective.
\end{thD}

\begin{proof}
Let $\Sigma\subset \boN$ be a properly immersed minimal $3$-punctured sphere in $\boN$.
Let $x_0\in \Sigma$ and $\alpha, \beta, \gamma$ be three loops at $x_0$
that are freely homotopic to embedded loops in the different ends of $\Sigma$. 

Let $\pi:\H^3\rightarrow \boN$ be a universal covering map and $\tilde x_0$ be
in $\pi^{-1}(x_0)$. Let $\widetilde\Sigma$ be the lift of $\Sigma$ passing
through $\tilde x_0$. The choice of $\tilde x_0$ induces a
monomorphism $\phi:\pi_1(\Sigma,x_0)\rightarrow \textrm{Isom}^+(\H^3)$.
$\widetilde\Sigma$ is then a proper immersion of the quotient of the universal
cover of $\Sigma$ by $\ker \phi$. Let $\Gamma$ be the image of the monomorphism
$\phi$. As a consequence $\widetilde\Sigma$ is properly immersed in $\H^3$.Let
us denote by $T_\alpha$ , $T_\beta$ and $T_\gamma$ the maps in $\Gamma$
associated by $\phi$ to $[\alpha],[\beta],[\gamma]$.

By the Finite Total Curvature Theorem \ref{th:ftc}, we know
each end is asymptotic to $\mu\times \R_+$ where $\mu$ is a geodesic in some
$T(c)$ in a cusp end of $\boN$; $\mu\times\R_+$ is a totally geodesic annulus in
this cusp end. The inclusion of $T(c)$ into $\boN$ induces an injection of the
fundamental group of $T(c)$ into that of $\boN$. Hence $\alpha, \beta$ and
$\gamma$ are sent to non zero parabolic elements of $\textrm{Isom}^+(\H^3)$ by $\phi$. 

Next we will prove the limit set of $\widetilde\Sigma$ is a circle $C$ in
$\partial_\infty\H^3\simeq\S^2$. Then the maximum principle yields that $\widetilde\Sigma$
is the totally geodesic plane $P$ bounded by $C$, thus proving Theorem~C. More
precisely, foliate $\H^3\cup\partial_\infty \H^3$ minus two points by totally
geodesic planes and their asymptotic boundaries so that $P$ is one leaf of the
foliation. This foliation at $\partial_\infty \H^3$ is a foliation by circles
with two "poles" $p$ and $q$. The circles close to $p$ bound hyperbolic planes
$Q$ in $\H^3$ that are disjoint from $\widetilde\Sigma$. As the circles in the
foliations of $\partial_\infty \H^3$ go from $p$ to $C$, there can be no first
point of contact of the planes with $\widetilde\Sigma$ ($\widetilde\Sigma$ is
proper and the limit set of $\widetilde\Sigma$ is $C$). Hence $\widetilde\Sigma$
is in the half space of $\H^3\setminus P$ containing $q$. The same argument with
planes coming from $q$ to $C$ shows that $\widetilde \Sigma=P$. So
$\widetilde\Sigma$ is simply connected which implies $\ker\phi=\{1\}$ and
$\Sigma$ is $\pi_1$
injective.

Let us now prove the existence of $C$.  We have the following claim
whose proof is based on Adams work \cite{Ada}.

\begin{claim}\label{cl:adams}
There is a circle $C$ in $\partial_\infty\H^3$ which is invariant by
$\Gamma$. The limit set of $\widetilde \Sigma$ is $C$.
\end{claim}

\begin{proof}
Using the half space model for $\H^3$ with
$\infty$ the fixed point of $T_\alpha$ and using the $SL_2(\C)$ representation
of $\textrm{Isom}^+(\H^3)$ we can write,
$$
T_\alpha=
\begin{pmatrix}1&w\\0&1\end{pmatrix}\quad\quad T_\beta=\begin{pmatrix}a&b\\ c&d
\end{pmatrix}
$$ 
with $w\in\C^*$, $a,b,c,d\in \C$ such that $ad-bc=1$ and $a+d=2$. Then
$T_\gamma=T_\alpha\cdot T_\beta=\begin{pmatrix}a+cw&b+dw\\ c & d\end{pmatrix}$
is parabolic so it must satisfy to $\lambda=a+cw+d=\pm 2$.

Since $a+d=2$ we have $c=0$ if $\lambda=2$ or $c=-4/w$ if $\lambda=-2$. If
$c=0$, $T_\alpha$, $T_\beta$ and $T_\gamma$ would fix the point $\infty$ and all
elements in $\Gamma$ would have $\infty$ as fixed point and $\{\infty\}$ is the
limit set of $\Gamma$. We will rule out this possibility below.

If $c=-4/w$, the fixed point of $T_\beta$ is $x_\beta= \frac{w(d-a)}8$ and the
fixed point of $T_\gamma$ is $x_\gamma=\frac{w(d-a)}8+\frac w2$. $T_\alpha$
leaves invariant the circle $C=\{\frac{w(d-a)}8+tw, t\in\R\}\cup \{\infty\}$.
Also we have 
\begin{align*}
T_\beta(\infty)&=-\frac{wa}{4}=\frac{w(d-a)}8=\frac w4\in C\\
T_\beta(x_\gamma)&=T_\alpha^{-1}(x_\gamma)=\frac{w(d-a)}8-\frac w2\in C
\end{align*}
Hence $T_\beta$ leaves $C$ invariant. Thus $\Gamma$ leaves $C$ invariant.
Actually, $C$ is the limit set of $\Gamma$.

Now let us see that the limit set of $\widetilde\Sigma$ is the limit set $\partial\Gamma$ of $\Gamma$. We
split $\Sigma$ in the union of a compact part $K$ containing $x_0$ and three 
cusp ends $C_i$. So $\widetilde\Sigma$ split in the union of the lift
$\widetilde K$ of $K$ and the union of pieces contained in disjoint horoballs
$H_\alpha$ with boundary along the horoball. Because of the asymptotic behavior of $\Sigma$,
the lift $g_i$ of $\partial C_i$ in $\partial H_\alpha$ is not homeomorphic to a
circle. Thus there is a non trivial $\gamma\in \Gamma$ that leaves $g_i$ and then
$H_\alpha$ invariant. So the center of $H_\alpha$ belongs to $\partial\Gamma$. 

If $\partial\Gamma$ is only one point (case $c=0$) it means that there is only
one horoball $H_\alpha$: it is $\{z\ge c\}$. Since any point in $\widetilde K$
is at a finite distance from $H_\alpha$ and $\widetilde K$ is periodic, the $z$
function reaches its minimum somewhere. The maximum principle then get a
contradiction. So $c\neq 0$ and $\partial \Gamma=C$.

Let $(p_i)$ be a proper sequence of points in $\widetilde\Sigma$ and assume it
converges to some point $p_\infty\in\partial_\infty \H^3$. If all the $p_i$
belong to $\widetilde K$, there is a sequence of elements $\gamma_i\in\Gamma$
such that the distance between $p_i$ and $\gamma_i\cdot \tilde x_0$ stays
bounded. So $(p_i)$ and $(\gamma_i\cdot \tilde x_0)$ have the same limit so
$p_\infty$ is in the limit set of $\Gamma$ so in $C$.

So we can assume that $p_i\in H_{\alpha_i}$ for all $i$. If the sequence $(\alpha_i)$ is
finite; $p_\infty$ is a center of one of the $H_\alpha$ so it is in $C$. If the
sequence $(\alpha_i)$ is not finite the the distance from $\tilde x_0$ to
$H_{\alpha_i}$ goes to $\infty$. There is a neighborhood $N_m$ of $C$ in
$\H^3\cup\partial_\infty\H^3$ that contains all horoballs centered on $C$ whose
distance to $\tilde x_\infty$ is
larger than $m$ and such that $\cap_{m>0} N_m= C$. So $p_\infty\in C$.
\end{proof}
\end{proof}

The proof of the preceding result is based on the study of the group
representation $\phi$. It can be controlled under some other hypotheses.

\begin{prop}
Let $S_0$ and $S_1$ be two properly immersed minimal surfaces in $\boN$ such
that the immersions are homotopic. If $S_0$ is totally geodesic, then $S_0=S_1$.
\end{prop}

\begin{proof}
Let $f_t:S\times[0,1]\rightarrow \boN$ be the homotopy between the two minimal
immersions. Let $x_0\in S$ and $\tilde{x_0}$ be a point in $\pi^{-1}(f_0(x_0))$
where $\pi:\H^3\rightarrow \boN$ is a covering map. Let $g_t:\widetilde
S\times[0,1]\rightarrow \boN$ be the lift of $f_t$ such that $g_0(x_0)=\tilde
x_0$. This defines a group representation $\phi:\pi_1( S\times[0,1],
(x_0,0))=\pi_1(S,x_0)\rightarrow \textrm{Isom}^+(\H^3)$ such that $g$ is
$\phi$-equivariant. Since $f_0(S)$ is totally geodesic, $g_0(\widetilde S)$ is a
totally geodesic disk. This implies that the image of $\phi$ has a circle $C$ as
limit set. As in the proof of Claim \ref{cl:adams}, it implies that
$g_1(\widetilde{S})$ has $C$ as asymptotic boundary. Then, as in the proof of
Theorem C, $g_1(\widetilde S)$ is totally geodesic and equals $g_0(\widetilde
S)$. So $f_0(S)=f_1(S)$.
\end{proof}

\section{Minimal surfaces in the cusp ends of $\boN$}
\label{sec:maxtrans}


In this section, we will analyse the behaviour of embedded minimal surfaces that
enter cusp ends of $\boN$. In dimension $2$, the situation is simple. If $N^2$
is a $2$-cusp (\textit{i.e.} a quotient of a horodisk of $\H^2$ by a parabolic
isometry leaving the horodisk invariant) then a geodesic that enters $N^2$
either goes straight to the cusp (\textit{i.e.} it is an orthogonal trajectory
of the horocycles of the cusp) or it leaves $N^2$ in a finite time. In dimension
$3$, for the moment we know that a properly immersed minimal annulus that enters
a cusp end of $\boN^3$ is asymptotic to a $2$-cusp of the end
($\gamma\times[c,\infty)$, $\gamma$ a compact geodesic of $T(c)$), or the
intersection of the minimal annulus with the end of $\boN$ is compact.

We will establish two maximum principles in the ends of $\boN$ which will
control the geometry of embedded minimal surfaces in the ends.

Let $C=M/G(v_1,v_2)$ be an end of $\boN$, parametrized by the quotient of
$M=\{(x,y,z)\in\R^3|z\ge 1/2\}$ as in Section \ref{sec:discussion}, with
$\Lambda(C)=\Lambda(v_1,v_2)$ the diameter of $T(1)$. Recall that we can make
$\Lambda(C)$ as small as we wish by passing to a subend $C'$ of $C$ defined by
$z\ge z_0$, $z_0$ large.

We modify the metric on $C$ introducing smooth functions
$\Psi:[1/2,\infty)\rightarrow \R$ satisfying $\Psi(z)=z$ for $1/2\le z\le 1$,
and $\Psi$ non decreasing. There will be other conditions on $\Psi$ as we
proceed.

Let $g_\Psi=\frac1{\Psi^2(X)}dX^2$ be a new metric on $C$; $g_\Psi$ is the
hyperbolic metric of $\boN$ for $1/2\ge z\ge 1$.

The mean curvature of the torus $T(z)$ in the metric $g_\Psi$ equals $\Psi'(z)$,
with respect to the unit normal $\Psi(z)\partial_z$ (so points towards
the cusp: perhaps it is zero). The sectional curvatures for $g_\Psi$ are:
$$
K_{g_\Psi}=\begin{cases}
-\Psi'(z)^2\text{ for the }(\partial_x,\partial_y)\text{ plane}\\
\Psi(z)\Psi''(z)-\Psi'(z)^2\text{ for the }(\partial_x,\partial_z)\text{ and }
(\partial_y,\partial_z)\text{ planes}
\end{cases}
$$

We will always introduce $\Psi$'s such that $|\Psi'|$ and $|\Psi\Psi''|$
are bounded by some fixed constant. Hence the sectional curvatures of the new
metrics will be uniformly bounded as well. Then given $\eps_0>0$, there is a
$k_0>0$ such that a stable minimal surface in $(C,g_\Psi)$ has curvature
bounded by $k_0$ at all points at least at a distance $\eps_0$ from the
boundary. The bound $k_0$ depends only on the bound of the sectional curvatures and
$\eps_0$ not on the injectivity radius \cite{RoSoTo}

\begin{remarq}\label{rem:isom2}
The pull back of the $g_\Psi$ by the map $H$ defined in Remark~\ref{rem:isom} is
$H^*{g_\Psi}=g_{\Psi_{z_0}}$ where $\Psi_{z_0}(z)=\frac1{z_0}\Psi(z_0z)$. This
modification does not change the estimates on $\Psi'$ and $\Psi\Psi''$.
\end{remarq}


\subsection{Maximum principles}
In the section we prove maximum principles for a cusp end $C$ endowed with a
metric $g_\Psi$. The following estimates will depend on an upper-bound on
$|\Psi'|$ and $|\Psi\Psi''|$.

We have a first result.

\begin{prop}[Maximum principle I]
Let $k_0,\eps_0>0$. There is a $\Lambda_0=\Lambda(k_0,\eps_0)$ such that if
$\Sigma$ is an embedded minimal surface in $(C,g_\Psi)$ with $|A_\Sigma|\le k_0$
and $\Lambda(C)\le \Lambda_0$. Then if $p\in \Sigma$ is at least an intrinsic
distance $\eps_0$ from $\partial\Sigma$ and if $z(q)\le z(p)$ for all $q$ in
the intrinsic $\eps_0$-disc centered at $p$, then $\Sigma=\{z=z(p)\}$ and
$\Psi'(z(p))=0$.
\end{prop}

\begin{proof}
Let $\pi: M\rightarrow C$ be the covering projection and
$\widetilde\Sigma=\pi^{-1}(\Sigma)$. We may suppose $p=\pi(0,0,z(p))$.

Since the curvature of $\widetilde\Sigma$ is bounded by $k_0$,
$\widetilde\Sigma$ is a graph of bounded geometry in a neighborhood of $p$.
Hence the exists $\mu=\mu(k_0,\eps_0)$ and a smooth function
$u:D_\mu(0,0)\rightarrow\R$, $D_\mu(0,0)=\{x^2+y^2\le \mu\}$, $u(0,0)=z(p)$ and
the graph of $u$ in $M$ is a subset of $\widetilde\Sigma$. $\mu$ can be chosen
such that, if $q\in \text{graph}(u)$, $d_{\Sigma}(\pi(q),p)<\eps_0$. Hence
$z(p)$ is a maximum value of $u$ in $D_\mu(0,0)$.

Now $\widetilde\Sigma$ is invariant by $G(v_1,v_2)$ so if $\Lambda_0<\mu/2$, we
have $v_1, v_2\in D_{\mu/2}(0,0)$. So $D_\mu(0,0)\cap D_\mu(v_1) = D \neq
\emptyset$.

Let $u_1:D\rightarrow \R$, be $u_1(q)=u(q-v_1)$; the graph of $u_1$ is contained
in $\widetilde \Sigma$. Then $u_1(v_1)=u(O)\ge u(v_1)$ since $u$ reaches its
maximum at $O=(0,0)$. Also $O\in D$, $u_1(O)=u(-v_1)\le u(O)$. Thus the graphs of
$u$ and $u_1$ over $D$ must intersect and since $\widetilde\Sigma$ is embedded,
$u=u_1$
on $D$. Hence $u_1$ is a smooth continuation of $u$ to $D_\mu(O)\cup
D_\mu(v_1)$. Repeating this with $G=\Z v_1+\Z v_2$, we see that $u$ extends
smoothly to an entire minimal graph contained in $\widetilde\Sigma$. This graph
is periodic with respect to $G$ hence bounded below. The maximum principle at a
minimum point of $u$ implies that $u$ is constant. Hence $u=u(0,0)=z(p)$ and
$T(z(p))$ is minimal so $\Psi'(z(p))=0$.
\end{proof}

Next we use the maximum principle I to prove a compact embedded minimal surface
can not go far into a cusp end; no \textit{a priori} curvature bound
assumed. More precisely we have the following statement.
 
\begin{prop}[Maximum principle II]\label{prop:max2} Let $0< t_0< 1/2$. There is a
$\Lambda_0=\Lambda_0(t_0)$ such that if $\Lambda(C)\le \Lambda_0$ and $\Sigma$
is a compact embedded minimal surface in $(C,g_\Psi)$ with $\partial\Sigma\subset
T(1-t_0)$ ($\Sigma$ being transverse to $T(1-t_0)$) then $\Sigma\subset\{z\le
1\}$.
\end{prop}

\begin{proof}
First suppose $\Sigma$ is a stable minimal surface. Then by curvature bounds for
stable surfaces \cite{RoSoTo}, we know there is a $k_0$ such that $|A_\Sigma|\le
k_0$ on $\Sigma\cap\{z\ge 1-t_0/2\}$; $k_0$ depends on our assumed bounds on
$\Psi', \Psi\Psi''$. By the maximum principle I, there is a $\Lambda_0$,
only depending on $t_0$, such that if $\Lambda(C)\le \Lambda_0$ then $z$ has no
maximum larger than $1$. Hence $\Sigma\subset\{z\le 1\}$.

Now suppose that $\Sigma$ is not stable. Choose $c_0$ and $c$ so that $z<c_0<c$
on $\Sigma$ and consider $\Sigma\subset X=\{1/2\le z \le c\}$. We remark that
$\Sigma$ separates $\inter{X}$, the interior of $X$. Indeed any loop in
$\inter{X}$ is homologous to a loop in $T(c_0)$ which does not intersect
$\Sigma$.  So the intersection number mod $2$ of a loop with $\Sigma$ is always
$0$. Then denote by $A$ the
connected component of $X\setminus \Sigma$ which contains $\{z=c\}$. \textit{A
priori} the boundary of $A$ is mean convex except for $\{z=c\}$. But we can
modify the function $\Psi$ for $c_0\le z\le c$ such that $\Psi'(c)=0$ and
keeping $\Psi$ non-decreasing and the bounds on $\Psi'$ and $\Psi\Psi''$ ($c$
should be assumed very large). Then $T(c)$ is now
minimal and $A$ has mean convex boundary. In $A$, there exist a least area
surface $\widetilde\Sigma$ with $\partial\Sigma=\partial\widetilde\Sigma$ . Now
the maximum of the $z$ function on $\widetilde\Sigma$ is larger than the one on
$\Sigma$. Since $\widetilde\Sigma$ is stable we already know that
$\widetilde\Sigma\subset\{z\le 1\}$, so $\Sigma$ as well.
\end{proof}


\subsection{Transversality}

Now we will see that embedded minimal surfaces of bounded curvature are
"strongly transversal" to $T(c)$ in $C$ endowed with the hyperbolic metric.

\begin{prop}\label{prop:transvers}
Let $k_0, \eps_0>0$ be given. There exist constants $\Lambda_0$
and $\theta_0$ such that if $\Sigma$ is an embedded minimal surface in
$(C,g_\H)$,
$\Lambda(C)\le \Lambda_0$, $|A_\Sigma|\le k_0$, with $\partial \Sigma$ at an
intrinsic distance distance greater than $\eps_0$ of the points of $\Sigma$ in
$T(1)$, then the angle between $\Sigma$ and $T(1)$ is at least $\theta_0$.
The constant $\Lambda_0$ and $\theta_0$ only depend on $k_0$ and $\eps_0$.
\end{prop}

\begin{proof}
If this proposition fails, there exists $\Sigma_n, p_n\in \Sigma_n\cap T(1)$ in
a hyperbolic cusp $C_n$ satisfying the hypotheses, such that $\Lambda(C_n)$ and
the angle between $\Sigma_n$ and $T(1)$ at $p_n$
goes to zero. Lift
$\Sigma_n$ to $M$ so that $p_n=(0,0,1)$. The curvature bound gives the
existence of a disk $D=D_\mu(0,0)\subset \R^2$ and smooth functions $u_n$ on $D$
whose graphs are contained in $\Sigma_n$ (for large $n$). These functions
have bounded $C^{2,\alpha}$ norm by the curvature bound and the fact that their
gradient at $(0,0)$ converges to zero. Hence a subsequence of the $u_n$
converges to a minimal graph $u$ over $D$ and the graph of $u$ is tangent to
$T(1)$ at $(0,0,1)$.

Let $v_1^n,v_2^n$ be the generator of the group leaving $C_n$ invariant. 
Let $v_0$ be in $D$. Since $\Lambda(C_n)\rightarrow 0$, there is a sequence
$(a_1^n,a_2^n)_{n\in\N}$ in $\Z^2$ such that
$a_1^nv_1^n+a_2^nv_2^n\rightarrow v_0$. The graph of
$u_n(\cdot-(a_1^nv_1^n+a_2^nv_2^n))$ over $D+a_1^nv_1^n+a_2^nv_2^n$ is also a
part of a lift of $\Sigma_n$. Since $\Sigma_n$ is embedded, its lift is also
embedded. So, for any $n$, we have either $u_n(\cdot)\le
u_n(\cdot-(a_1^nv_1^n+a_2^nv_2^n))$ or $u_n(\cdot)\ge 
u_n(\cdot-(a_1^nv_1^n+a_2^nv_2^n))$. Thus at the limit, $u(\cdot)\le u(\cdot-v_0)$
or $u(\cdot)\ge u(\cdot-v_0)$ on $D\cap (D+v_0)$.

Let $S$ be the totally geodesic surface in $M$ tangent to $\{z=1\}$ at
$(0,0,1)$. Over $D$, $S$ can be described as the graph of a radial function $h$.
We have $h(0,0)=1$ and there is $\alpha>0$ such that, over $D$, $h((x,y)\le
1-\alpha(x^2+y^2)$. The functions $u$ and $h$ are two solutions of the minimal
surface equation with the same value and the same gradient at the origin. So the function
$u-h$ looks like a harmonic polynomial of degree at least $2$.

If the degree is $2$, one can find $v_0\in D\setminus\{(0,0)\}$ such that
$(u-h)(v_0)<0 $ and $(u-h)(-v_0)<0$. Then we have
\begin{gather*}
u(v_0)<h(v_0)<h(0,0)=u(v_0-v_0)\\
u((0,0)-v_0)<h(-v_0)<h(0,0)=u(0,0)
\end{gather*}
This contradicts $u(\cdot)\le u(\cdot-v_0)$ or $u(\cdot)\ge u(\cdot-v_0)$ on the
whole $D\cap (D+v_0)$.

If the degree is larger than $3$, the growth at the origin of $h$ implies that
there is a disk $D'$ centered at the origin included in $D$ such that $u<1$ in
$D'\setminus\{(0,0)\}$. So if $v_0\in D'\setminus\{(0,0)\}$ we have
$$
u(v_0)<u(0,0)=u(v_0-v_0) \text{ and } u((0,0)-v_0)=u(-v_0)<u(0,0)
$$
This gives also a contradiction $u(\cdot)\le u(\cdot-v_0)$ or $u(\cdot)\ge
u(\cdot-v_0)$ on the whole $D\cap (D+v_0)$.
\end{proof}

We notice that Remark~\ref{rem:isom} can be used to get strong transversality
with $T(c)$ for $c\ge 1$.

A consequence of Proposition~\ref{prop:transvers} is then the following result.
\begin{thC}
Let $\Sigma$ be a properly embedded minimal surface in $\boN$ of bounded
curvature. Then $\Sigma$ has finite topology.
\end{thC}
\begin{proof}
If $k_0$ is an upper bound of the norm of the second fundamental form of
$\Sigma$, Proposition~\ref{prop:transvers} gives a constant $\Lambda_0$. Now
$\boN$ can be decomposed as the union of a compact part $K$ and a finite number
of cusp-ends $C_i$ with $\Lambda(C_i)\le\Lambda_0$. Since $\Sigma$ is
transversal to the tori $T_i(c)$, $\Sigma$ has the same topology as
$\Sigma\cap\inter{K}$; so it has finite topology.
\end{proof}


\section{Existence of compact embedded minimal surfaces in $\boN$}
\label{sec:compact}


Producing minimal surfaces is often done by minimizing the area in a certain
class of surfaces. In order to ensure the compactness of our surface
in $\boN$, a min-max argument is more suitable in our proof of the following
existence result.

\begin{thA}
There exists a compact embedded minimal surface in any $\boN$.
\end{thA}

\begin{proof}
Let $C_1,\cdots,C_k$ be the cusp ends of $\boN$. Let $z_i$ be the
$z$-coordinates in $C_i$ and assume that $\Lambda(C_i)\le \Lambda_0$, for
$1\le i\le k$; $\Lambda_0$ the constant given by the maximum principle II.
This can always be realized by Remarks~\ref{rem:isom} and \ref{rem:isom2}.

Now we change the hyperbolic metric in each end $C_i$ as follows. Let
$\Psi:[1/2,\infty)\rightarrow \R$ satisfy $\Psi(z)=z$ for $1/2 \le z\le 1$,
$\Psi'(z)>0$ and $\lim_{z\rightarrow \infty} \Psi(z)=3/2$.

Let $L$ be large (to be specified later) and modify the metric $g_\Psi$ in
$[L,L+1]$ so the new metric gives a compactification of $C_i$ by removing
$\{z_i\ge L\}$ and attaching a solid torus to $T(L)$. The precise way to do this
will be explained below. With this new metric, for $L\le z\le L+1$ the mean
curvature of the tori $T_i(z)$ is increasing; going from $\Psi'(L)$ at $z=L$ to
$\infty$ as $z\rightarrow L+1$. $z=L+1$ corresponds to the core of the solid
torus. We do this in each cusp and get a compact manifold $\tilde\boN$ without
boundary endowed with a certain metric. We notice that the manifold does not
depend on $L$ but the metric does.

Now we can choose a Morse function $f$ on $\tilde\boN$ such that all the
tori $T_i(z)$, $1/2\le z\le L$, $1\le i\le k$, are level surfaces of $f$.

This Morse function $f$ defines a sweep-out of the manifold $\tilde \boN$ and 
$$
M_0=\max_{t\in\R}\boH^2(f^{-1}(t))
$$
essentially does not depend on $L$ (in fact it can decrease when $L$ increases)
($\boH^2$ is the $2$-dimensional Hausdorff measure).

Almgren-Pitts min-max theory applies to this sweep-out and gives a compact embedded
minimal surface $\Sigma$ in $\tilde\boN$ whose area is at most $M_0$ (see
theorem 1.6 in \cite{CoDeL}). Let us see now that $\Sigma$ actually lies in the
hyperbolic part of $\tilde\boN$ so in $\boN$.

Since $\Psi\rightarrow 3/2$, the metric on $T_i(k,k+1)$ is uniformly controlled
and close to being flat. As a consequence of the monotonicity formula for minimal
surfaces (see Theorem 17.6 in \cite{Sim}), if $\Sigma\cap
T_i(k+1/2)\neq\emptyset$ ($1\le k\le L-1$), the area of $\Sigma\cap T_i(k,k+1)$
is at least $c_0>0$. The constant $c_0$ only depends on the ambient sectional
curvature bound and the $v_j^i$'s (the vectors in the end $C_i$).

This monotonicity formula gives at least linear growth for $\Sigma$. More
precisely, if a connected component of $\Sigma$ intersect $T_i(1)$ and $T_i(L)$
is has area at least $c_0(L-1)$. So by choosing, $L$ larger than $M_0/c_0$ there
is no component of $\Sigma$ meeting both $T_i(1)$ and $T_i(L)$.

Also, no connected component lies entirely in $\{z_i\ge 1\}$. Indeed, the
$z_i$ would have a minimum on the component which is impossible by the classical maximum
principle and the sign of the mean curvature on $T_i(z)$. Thus $\Sigma$ stays
out of $\{z_i\ge L\}$. Hence by the maximum principle II, $\Sigma$ does not
enter in any $\{z_i\ge 1\}$ which completes the proof.

Let us now give the definition of the new metric on $[L,L+1]$. The tori $T_i(c)$
are the quotient of $\R^2$ by $v_1^i,v_2^i$ so they can be parametrized by
$u\frac{v_1^i}{2\pi}+v\frac{v_2^i}{2\pi}$ where $(u,v)\in\S^1\times\S^1$. With
this parametrization, the metric $g_\Psi$ on $C_i$ is then 
\begin{multline*}
\frac1{4\pi^2\Psi(z_i)^2}(|v_1|^2\dd u^2+2(v_1,v_2)\dd u\dd v+|v_2|^2\dd v^2+d z_i^2)=\\
\frac1{\Psi(z_i)^2}(a^2\dd u^2+2b\dd u\dd v+c^2\dd v^2+dz_i^2)
\end{multline*}
Let $\phi$ be a smooth non increasing function on $[L,L+1]$ such that $\phi(z)=1$
near $L$ and $\phi(z)=((L+1)-z)/a$ near $L+1$. We then change the metric on
$\{L\le z_i\le L+1\}$ by
\begin{equation}\label{eq:metric}
\frac1{\Psi^2(z_i)}(\dd z_i^2+a^2\phi(z_i)^2\dd u^2+2b\phi(z_i)\dd u\dd v+c^2\dd v^2)
\end{equation}
Actually, this change consists in cutting $\{z_i\ge L\}$ from the cusp end $C_i$
and gluing a solid torus along $T(L)$. To see this, let $D$ be the unit disk with
its polar coordinates $(r,\theta)\in[0,1]\times\S^1$ and let us define the map
$h:D\times \S^1\rightarrow \S^1\times \S^1\times[L,L+1]$ by
$(r,\theta,v)\mapsto(\theta,v,L+1-r)$. The induced metric by $h$ from the one in
\eqref{eq:metric} for $r$ near $0$ is
\begin{multline*}
\frac1{\Psi^2(L+1-r)}(\dd r^2+a^2\frac{r^2}{a^2}\dd \theta^2+2b\frac r a\dd
\theta\dd v+c^2\dd v^2)\\=\frac1{\Psi^2(L+1-r)}(\dd r^2+r^2\dd \theta^2+2\frac b
a r\dd \theta\dd v+c^2\dd v^2)
\end{multline*}
This is a well defined metric on the solid torus $D\times \S^1$.

With this new metric, the tori $T_i(c)=\{z_i=c\}$ ($c\in [L,L+1)$) have constant
mean curvature $\Psi'(c)-\frac{\phi'(c)}{2\phi(c)}\Psi(c)>0 $ with respect to
$\Psi(z)\partial_z$.
\end{proof}

A minimization argument can be done under some hypotheses to produce compact
minimal surfaces.

\begin{thB}
Let $S$ be a closed orientable embedded surface in $\boN$ which is not a
$2$-sphere or a torus. If $S$ is incompressible and non-separating, then $S$ is
isotopic to a least area embedded minimal surface.
\end{thB}

\begin{proof}
Let $C_1,\cdots,C_k$ be the cusp ends of $\boN$. Let $z_i$ be the
$z$-coordinates in $C_i$ such that the surface $S$ does not enter in $\{z_i\ge
1\}$. We assume that $\Lambda(C_i)\le \Lambda_0$, for $1\le i\le k$; $\Lambda_0$
the constant given by the maximum principle II for the function $\Psi$ below.

Let $\Psi:\R_+^*\rightarrow \R$ be a smooth increasing function such that
$\Psi(z)=z$ on $(0,1]$ and $\Psi'(2)=0$.

For each $a\ge 1$, let $\boN(a)$ be $\boN$ with each cusp end truncated at
$z_i=a$; \textit{i.e.} $\boN(a)=\boN\setminus\cup_{1\le i\le k}\{z_i>a\}$. We
remark that the $\boN(a)$ are all diffeomorphic to each other. 

Let $n$ be an integer. In each cusp end $C_i$, we change the metric on
$\boN(2n)$ by using a function $\Psi_n:[1/2,2n]\rightarrow \R;
z\mapsto n\Psi(\frac z{n})$. So $\Psi_n(z)=z$ on $[1/2,n]$ and
$\Psi'_n(2n)=0$; the torus $T_j(2n)$ minimal. We notice that the metric on
$\boN(n)$ is not modified.

Let us minimize the area in the isotopy class of $S$ in the manifold with
minimal boundary $\boN(2n)$. By Theorem~5.1 and remarks before Theorem~6.12 in
\cite{HaSc}, there is a least area surface $\Sigma_n$ in $\boN(2n)$ which is
isotopic to $S$. Theorem~5.1 in \cite{HaSc} can be applied because $\boN(2n)$ is
$P^2$-irreducible ($\boN(2n)$ is orientable and its universal cover is
diffeomorphic to $\R^3$). Moreover the minimization process does not produce a
non-orientable surface since, in that case, $S$ would be isotopic to the
boundary of the tubular neighborhood of it, hence $S$ would separate $\boN$. Finally $\Sigma_n$ is not one connected component of $\partial\boN(2n)$ since $S$ is not a torus.

In $\boN(2n)$, $(\{z=c\})_{c\in[1,2n]}$ is a mean convex foliation so
$\Sigma_n\cap\boN(1)\neq \emptyset$. By the maximum principle II, it implies
that $\Sigma_n\subset \boN(1)$ so in a piece of $\boN$ where the metric never
changes. \textit{A priori}, the surfaces $\Sigma_n$ could be different. But,
since they all lie in $\boN(1)$, they all appear in the minimization process in
$\boN(2)$ so they all have the same area. So $\Sigma_1$ is a least area surface
in the isotopy class of $S$ in $\boN$ with the hyperbolic metric.
\end{proof}

\begin{remarq}\label{rem:estim}
We can notice that there is a uniform lower bound for the area of minimal
surfaces in manifolds $\boN$. The point is that the thick part of such a
manifold $\boN$ is not empty. So at each point in the thick part there is an
embedded geodesic ball of radius $\eps_3/2$ centered at that point where
$\eps_3$ is the Margulis constant of hyperbolic $3$-manifolds. 

Each connected component of the thin part is either a hyperbolic cusp or the
tubular neighborhood of a closed geodesic. So it is foliated by mean convex
surfaces and a minimal surface $\Sigma$ can not be included in such a component.
So there is $x\in\Sigma$ in the thick part. Thus we can apply a monotonicity
formula (see \cite{Sim}) to conclude that the area of the part of $\Sigma$
inside the geodesic ball of radius $\eps_3/2$ and center $x$ is larger than a
constant $c_3>0$ that depends only on the geometry of the hyperbolic $\eps_3/2$
ball.

When the compact minimal surface $\Sigma$ is stable, we can be more precise (see
\cite{Has} where Hass attributes this estimates to Uhlenbeck). Applying the
stability inequality to the constant function $1$. We get that 
$$
\int_\Sigma -(\Ric(N,N)+|A|^2)\ge 0.
$$
Since $|A|^2=-2(K_\Sigma+1)$, the Gauss-Bonnet formula gives
$$
Area(\Sigma)\ge -\frac12\int_\Sigma K_\Sigma=-\frac12 \chi(\Sigma)=2\pi(g-1)
$$
with $g$ the genus of $\Sigma$. Morover, by Gauss formula $K_\Sigma\le -1$, so
the Gauss-Bonnet formula gives $Area(\Sigma)\le 4\pi(g-1)$.
\end{remarq}


\section{Existence of non compact embedded minimal surfaces in $\boN$}
\label{sec:noncompact}


In \cite{HaRuWa}, Hass, Rubinstein and Wang construct proper minimal surfaces in
manifolds $\boN$ by a minimization argument in homotopy classes. In \cite{Rub},
Ruberman constructs least area surfaces in the isotopy class. Here we make use
of results in Section \ref{sec:maxtrans} to give a different approach on the
proof of this second result.

First we remark that, in manifolds $\boN$, there is always a "Seifert" surface.
$\boN$ is topologically the interior of a compact manifold $\barre{\boN}$ with
tori boundary components and each boundary torus is incompressible. By Lemma 6.8
in \cite{Hem}, there is a compact embedded surface $\barre{S}$ in $\barre{\boN}$
with non empty boundary which is incompressible and $2$-sided; moreover it is
non-separating. Then $S=\barre{S}\cap \boN$ is a properly embedded smooth
surface in $\boN$, $S$ incompressible, of finite topology, non compact,
non-separating and $2$-sided.

The result is the following statement.

\begin{thm}\label{th:isotopy}
Let $S$ be a properly embedded, non compact, finite topology, incompressible,
non separating surface in $\boN$. Then $S$ is isotopic to a least area embedded
minimal surface.
\end{thm}

\begin{proof}
$S$ has a finite number of annular ends $A_1,\cdots, A_p$, each one being
included in one cusp end $C_i$ of $\boN$. Since $A_j$ is incompressible in
$C_i$, we can isotope $S$ so that each annular end $A_j$ is totally geodesic in
the end $C_i$ it enters. We still call $S$ this new surface and we notice that its
area is finite for the hyperbolic metric.

Let $\Psi:\R_+^*\rightarrow \R$ be a smooth increasing function such that
$\Psi(z)=z$ on $(0,1]$ and
$\Psi'(4/3)=0$. Let $\Lambda_0$ be the constant given by the Maximum principle
II and the transversality lemma (Propositions~\ref{prop:max2} and
\ref{prop:transvers}). Assume the ends of $\boN$ are chosen so that
$\Lambda(C_i)\le \Lambda_0$ for each end $C_i$.

As in the proof of Theorem B, we denote $\boN(a)=\boN\setminus\cup_{1\le i\le
k}\{z_i>a\}$. We remark that $\boN(a)$ is diffeomorphic to $\barre\boN$.

Let $n$ be a large integer. In each cusp end $C_i$, we change the metric on
$\boN(4n)$ by using a function $\Psi_n:[1/2,4n]\rightarrow \R;
z\mapsto 3n\Psi(\frac z{3n})$. So $\Psi_n(z)=z$ on $[1/2,3n]$ and
$\Psi'_n(4n)=0$; the torus $T_j(4n)$ minimal. We notice that the metric on
$\boN(3n)$ is not modified.

Let $S(4n)=S\cap \boN(4n)$; the area of $S(4n)$ is bounded by some constant $A$
independent of $n$. By Theorem~6.12 in \cite{HaSc}, there is a least area
surface $\Sigma(4n)$ in $\boN(4n)$, isotopic to $S(4n)$ and
$\partial\Sigma(4n)=\partial S(4n)$. We remark that $\Sigma(4n)$ is stable so
has bounded curvature away from its boundary (independent of $n$) (see
\cite{RoSoTo}).

In each cusp end $C_i$, Proposition~\ref{prop:transvers} implies $\Sigma(4n)$ is
transverse to the tori $T_1(a)$, $1\le a\le 2n$ (see Remark~\ref{rem:isom}, in
order to apply Proposition~\ref{prop:transvers}). So each intersection
$\Sigma(4n)\cap T_1(a)$ is composed of the same number of Jordan curves for
$1\le a \le 2n$. The next claims prove that this number is equal to the
number of boundary components of $\Sigma(4n)$ on $T_1(4n)$.

\begin{claim}\label{cl:1}
Let $\Ome$ be a domain in $\Sigma(4n)$ with boundary in $T_1(a)$ ($1\le
a\le n$). Then $\Ome$ does not enter in any $\{z_i\ge a\}$.
\end{claim}
\begin{proof}
If $\Sigma$ enters in one $\{z_i\ge a\}$, by transversality, it enters in
$\{z_i\ge 2n\}$. So the function $z_i$ will have a maximum larger than $2n$
which is impossible by Proposition~\ref{prop:max2} (see also
Remark~\ref{rem:isom}).
\end{proof}

\begin{claim}\label{cl:2}
Let $\gamma$ be a connected component of $\Sigma(4n)\cap T_1(a)$ ($1\le a\le
n$). Then $\gamma$ is not trivial in $\pi_1(T_1(a))$.
\end{claim}

\begin{proof}
Assume that $\gamma$ is trivial in $\pi_1(T_1(a))$. Since $\Sigma(4n)$ is
incompressible, $\gamma$ bounds a disk $\Delta$ in $\Sigma(4n)$. By
Claim~\ref{cl:1}, $\Delta$ stays in $\boN(a)$ where the metric is still
hyperbolic. So we can lift $\Delta$ to a minimal disk $\Delta'$ in
$\R^2\times\R_+$ (with the hyperbolic metric) with boundary in $z_1=a$ and
entirely included in $\{z_1\le a\}$. This is impossible by the maximum principle
since $\{z_1=s\}$ has constant mean curvature one.
\end{proof}

\begin{claim}
Let $\Sigma$ be a connected component of $\Sigma(4n)\cap\{n\le z_1\le 4n\}$.
Then $\Sigma$ is an annulus with one boundary component in $T_1(n)$ and one
in $T_1(4n)$.
\end{claim}

\begin{proof}
Let us first prove that the inclusion map of $\Sigma$ in $\{n\le z_1\le 4n\}$ is
$\pi_1$-injective. So let $\gamma$ be a loop in $\Sigma$ which bounds a disk in
$\{n\le z_1\le 4n\}$. Since $\Sigma(4n)$ is incompressible, there is a disk
$\Delta$ in $\Sigma(4n)$ bounded by $\gamma$. If $\Delta$ is in $\Sigma$, we
are done. If not, there is a subdisk $\Delta'$ of $\Delta$ with boundary in
$T_1(n)$; but this is impossible by Claim~\ref{cl:2}. So the inclusion map is
$\pi_1$-injective. We notice that $\pi_1(\{n\le z_1\le 4n\})$ is
Abelian, so $\pi_1(\Sigma)$ is Abelian. This implies that $\Sigma$ is topologically a sphere,
a disk, an annulus or a torus. The sphere and the torus are not possible since
$\Sigma$ has a non-empty boundary. Claim~\ref{cl:1} implies that
$\Sigma$ must have a boundary component on $T_1(4n)$. If the whole boundary of
$\Sigma$ is in $T_1(4n)$, the $z_1$ function admits a minimum on $\Sigma$
that is impossible by the maximum principle since the $T_1(c)$ have positive
mean curvature. So $\Sigma$ is an annulus with one boundary component in
$T_1(n)$ and one in $T_1(4n)$.
\end{proof}

With these claims, we have thus proved that $\Sigma(4n)\cap \boN(n)$ is isotopic
to $S\cap\boN(n)$ (here, we allow the boundary to move). We also notice that
because of the curvature estimate on $\Sigma(4n)$ and the transversality
estimate given by Proposition~\ref{prop:transvers}, the intersection curves
$\Sigma(4n)\cap T_1(a)$ ($1\le a\le n$) have bounded curvature. So they have a
well controlled geometry far in the cusp. More precisely, there is $a_0$ such
that $\Sigma(4n)\cap \{a_0\le z_1\le n\}$ is a graph over $S\cap\{a_0\le z_1\le
n\}$. So the sequence $\Sigma(4n)\cap \boN(n)$
is a sequence of surfaces with uniformly bounded area and curvature whose
behavior in the cusps is well controlled. Thus a subsequence converges to
a minimal surface $\Sigma$. This convergence says that $\Sigma(4n)\cap \boN(k)$
can be written as a graph or a double graph over $\Sigma\cap\boN(k)$. In the
first case, the surface $\Sigma$ is then isotopic to $S$. In the second case,
$\Sigma(4n)\cap\boN(k)$ is isotopic to the boundary of a tubular neighborhood of
$\Sigma\cap \boN(k)$ in $\boN(k)$; this implies that $\Sigma(4n)\cap\boN(k)$ is
a separating surface which is impossible by the properties of $S$.
\end{proof}

We notice that the area estimate given in Remark~\ref{rem:estim} are also true
for non compact minimal surface. Indeed, because of the asymptotic behaviour of
a stable minimal surface, the constant function $1$ can be used as a test
function even in the non compact case.

\section{Some examples}
\label{sec:examples}

In this section, we give some "explicit" examples that illustrate the above
theorems.

H. Schwarz and A. Novius constructed periodic minimal surfaces in $\R^3$ by
constructing minimal surfaces in a cube possessing the symmetries of the cube.
These surfaces then extend to $\R^3$ by symmetry in the faces. 

K. Polthier constructed periodic embedded minimal surfaces in $\H^3$ in an
analogous manner. Let $P$ be a finite side polyhedron of $\H^3$ such that
symmetry in the faces of $P$ tessellate $\H^3$. If $\Sigma_0$ is an embedded
minimal surface in $P$, meeting the faces of $P$ orthogonally and with the same
symmetry as $P$. Then $\Sigma$ extends to an embedded minimal surface in $\H^3$
by symmetry in the faces. Polthier makes this work for many polyhedron $P$;
\textit{e.g.} for all the regular ideal Platonic solids whose vertices are on
the spheres at infinity. Among these examples, one can obtain examples in
complete hyperbolic $3$-manifolds of finite volume.

We first describe how this technique yields an embedded genus $3$ compact
minimal surface in the figure eight knot complement $\boN$. 

Let $T$ be an ideal regular tetrahedron of $\H^3$; all the dihedral angles are
$2\pi/3$. In the Klein model of $\H^3$ (the unit ball of $\R^3$), $T$ is a
regular Euclidean tetrahedron with its four vertices on the unit sphere. Label
the faces of $T$ and two vertices of $T$, as in Figure~\ref{fig:mani}. Then identify face $A$
with face $B$ by a rotation by $2\pi/3$ about $v$ , and identify $D$ with $C$ by
a rotation by $2\pi/3$ about $w$.

\begin{figure}[h]
\subfloat[\label{fig:mani}]{\resizebox{0.5\linewidth}{!}{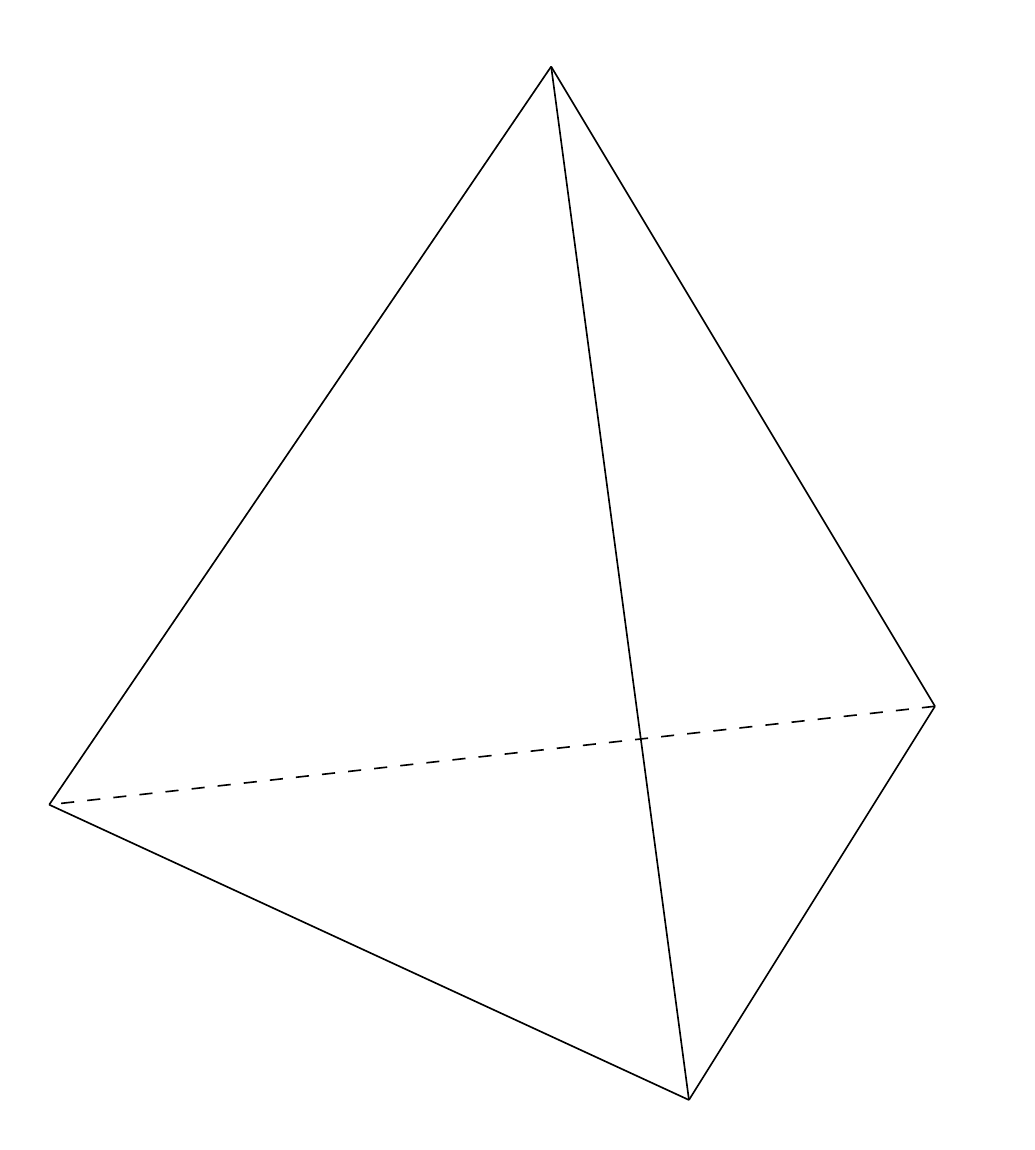}}
\subfloat[\label{fig:tessel}]{\resizebox{0.5\linewidth}{!}{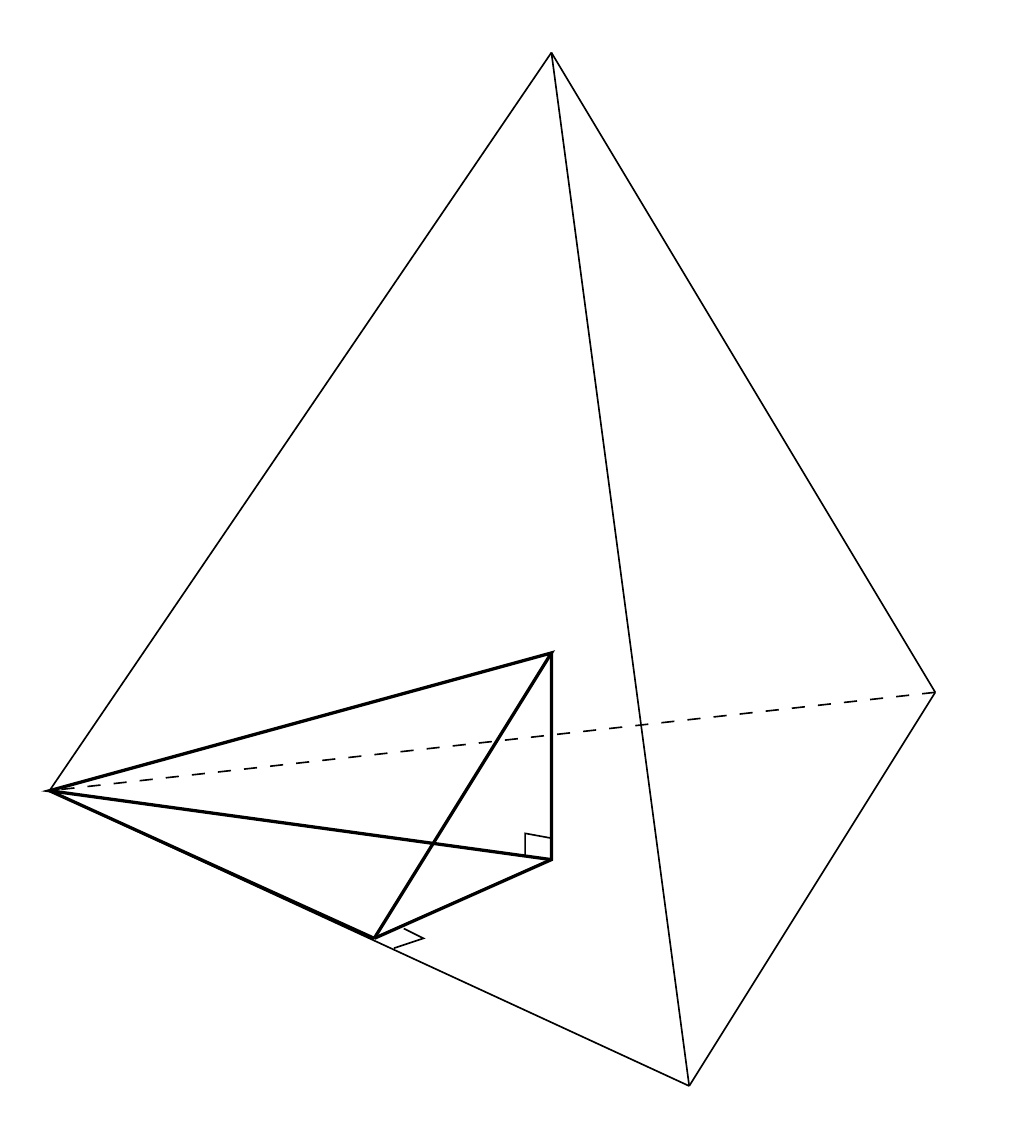}}\\
\begin{center}
\subfloat[\label{fig:polthier}]{\resizebox{0.4\linewidth}{!}{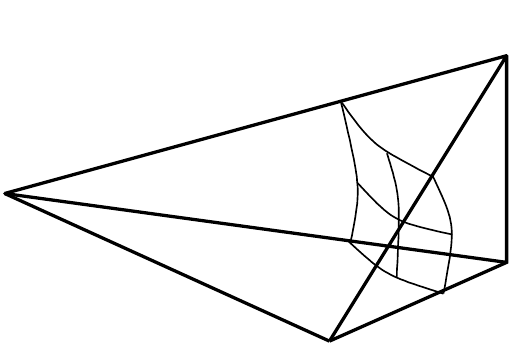}}
\end{center}
\caption{a minimal surface in the Giesekind manifold}
\end{figure}

The quotient of $T$ by these face matchings, produces a non orientable
hyperbolic $3$-manifold of finite volume. There is one vertex and its link is a
Klein bottle. This manifold $\boN$ was discovered by Giesekind in 1912.

The orientable $2$-sheeted cover $\boN'$ of the Giesekind manifold is
diffeomorphic to the complement of the figure eight knot in $\S^3$ ; hence is a
complete hyperbolic manifold of finite volume. In \cite{Thu}, Thurston explains how
$\boN'$ is homeomoprhic to the complement of the figure eight knot (see also
\cite{Fra}).

We construct an embedded compact minimal surface in $\boN$ that lifts to a
surface of genus $3$ in $\boN'$.

The geodesics from each vertex of $T$ to its opposite face, all meet at one
point $p$ in $T$. Join $p$ to each edge of $T$ by the minimizing geodesic. Also
join $p$ to each vertex of $T$ by a geodesic. This produces the edges of a
tessellation of $T$ by $24$ congruent tetrahedra.

Consider the tetrahedron $T_1$ of this tessellation as in
Figure~\ref{fig:tessel}. By a conjugate
surface technique, Polthier proved there exists an embedded minimal disk $D_1$
in $T_1$ meeting the boundary of $T_1$ orthogonally as in
Figure~\ref{fig:polthier}. Symmetry by
the faces of $T_1$ (and the faces of the symmetric tetrahedron of the
tessellation of $T$) extend $D_1$ to an embedded minimal surface $S$ meeting
each face of $T$ in one embedded Jordan curve in the interior of the face. $S$
is topologically a sphere minus $4$ points.

The face identification on $T$ send $S\cap A$ to $S\cap B$ and $S\cap D$ to
$S\cap C$. Hence $S$ passes to the quotient in $\boN$ to a compact embedded
minimal surface whose topology is the connected sum of two Klein bottles. The
lift of this to $\boN'$ is a genus $3$ compact embedded minimal surface.

A Seifert surface for the figure eight knot is an incompressible surface
homeomorphic to a once punctured torus. Applying Theorem~\ref{th:isotopy} gives
a properly embedded minimal one punctured torus in the complement of the figure
eight knot. 

Theorem $4$ of C.~Adams \cite{Ada} yields many totally geodesic properly
embedded $3$-punctured spheres in complete hyperbolic $3$-manifolds $\boN$ of
finite volume. Suppose $\boN$ arises as a link or knot complement that contains
an embedded incompressible $3$-punctured sphere (so by Adams, it is isotopic to
a totally geodesic one). For example if the link or the knot contains a part as
in Figure~\ref{fig:disk} such that the disk $D$ with $2$ punctures is a $3$-punctured
incompressible sphere in $\boN$. An example is the Whitehead link
(Figure~\ref{fig:whitehead}).

\begin{figure}[h]
\begin{center}
\subfloat[\label{fig:disk}]{\resizebox{0.4\linewidth}{!}{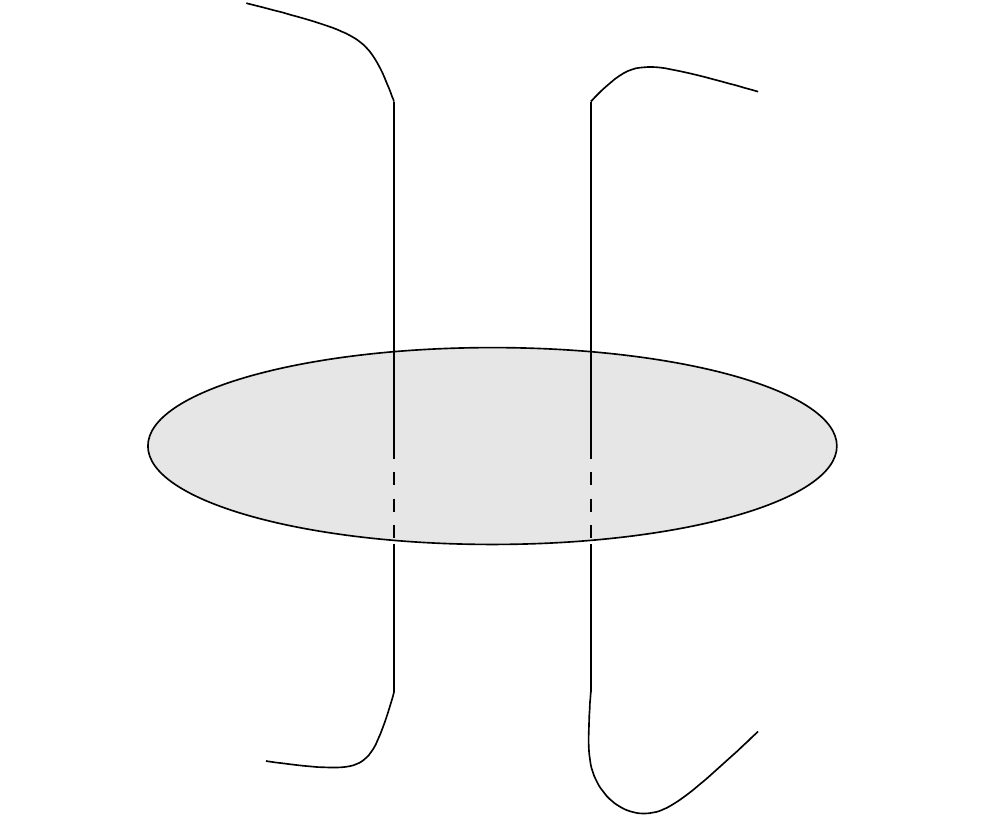}}
\subfloat[\label{fig:whitehead}]{\resizebox{0.4\linewidth}{!}{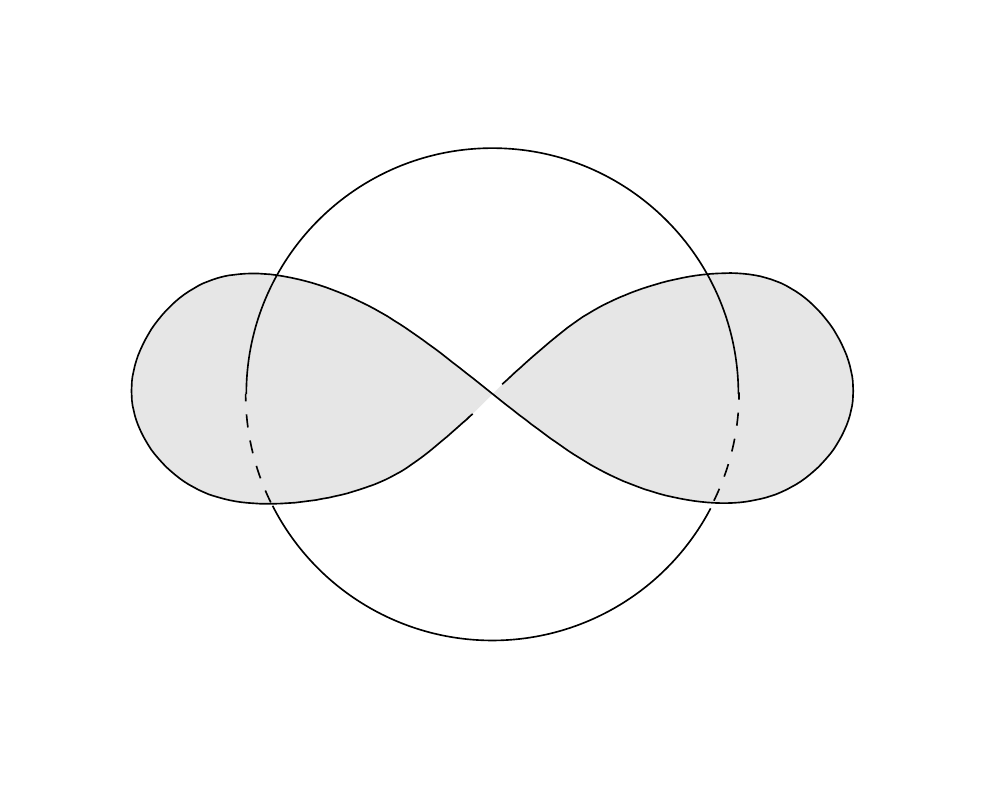}}
\caption{Incompressible $3$-punctured sphere in general position and in the
complement of Whitehead link}
\end{center}
\end{figure}

The Borromean rings is also a hyperbolic link. Its complement contains an embedded
incompressible thrice punctured sphere (Figure~\ref{fig:3spher}) and an embedded
once punctured torus (Figure~\ref{fig:1torus}) which is isotopic to a properly
embedded minimal once punctured torus by Theorem~\ref{th:isotopy}.

\begin{figure}[h]
\begin{center}
\subfloat[\label{fig:3spher}]{\resizebox{0.4\linewidth}{!}{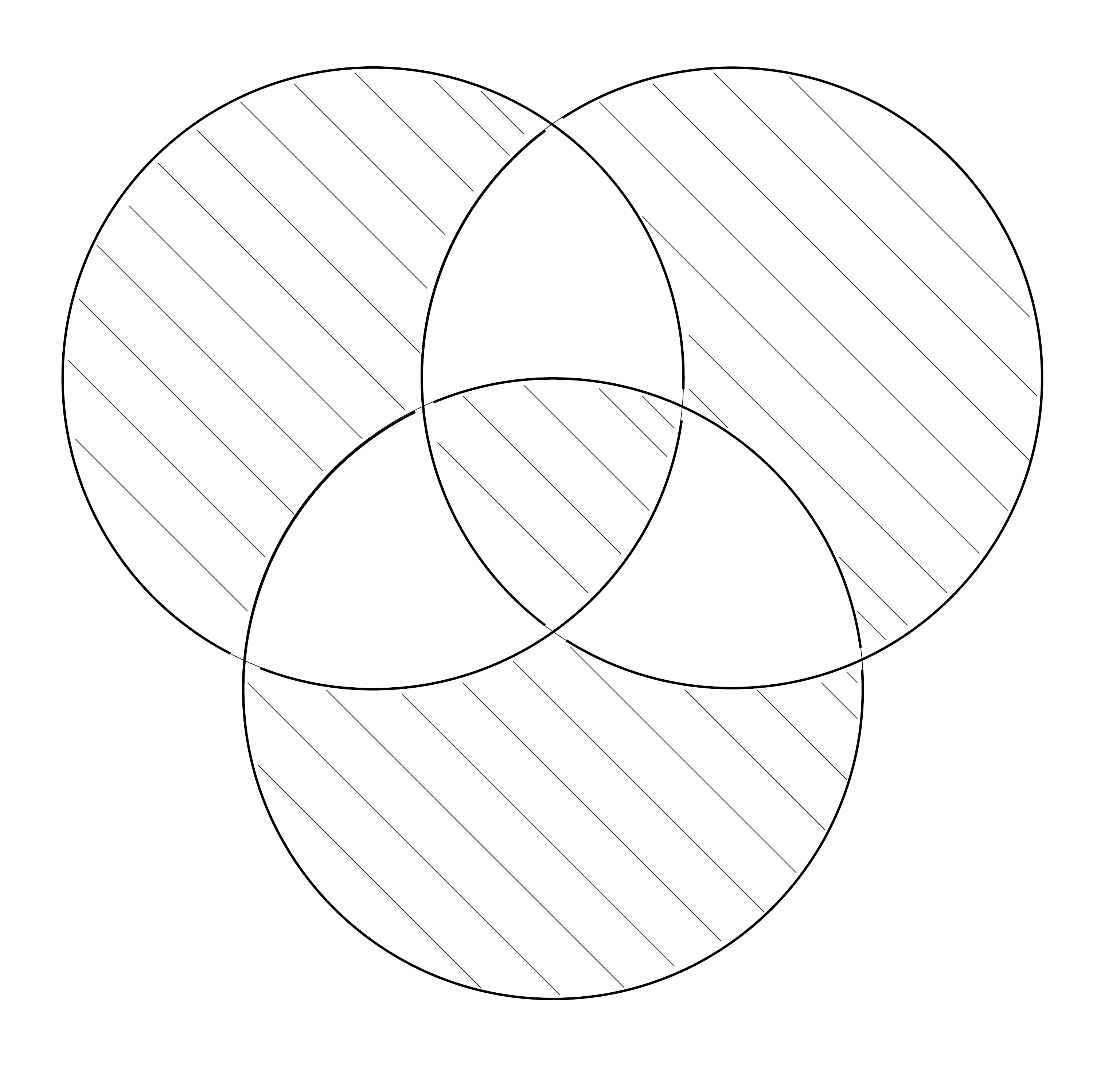}}
\subfloat[\label{fig:1torus}]{\resizebox{0.4\linewidth}{!}{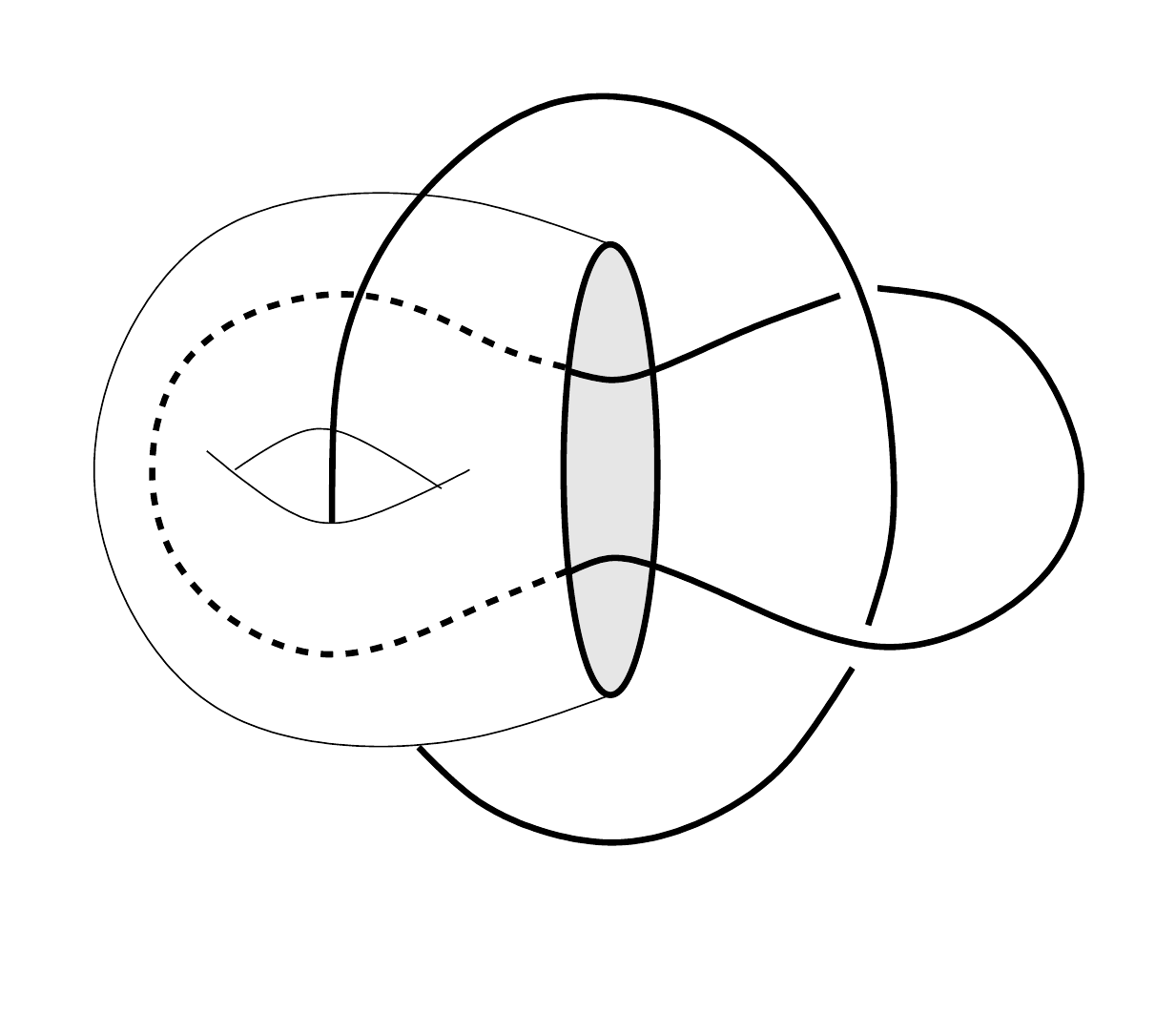}}\\
\subfloat[\label{fig:eight}]{\resizebox{0.6\linewidth}{!}{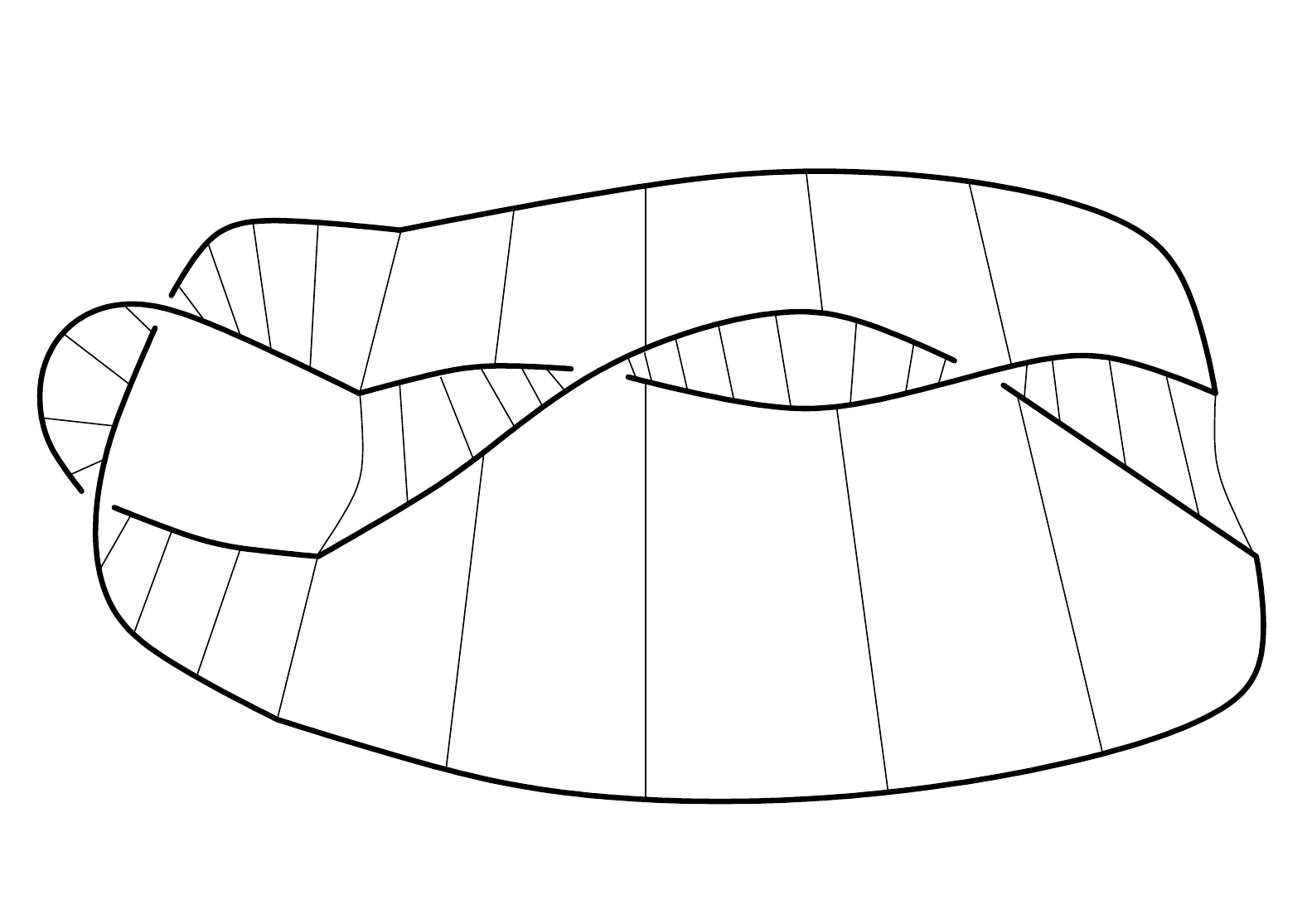}}
\caption{Incompressible $3$-punctured sphere and $1$-punctured torus in the
complement of Borromean rings and an incompressible $1$-punctured torus in the
figure eight knot complement}
\end{center}
\end{figure}

It will be interesting to estimate the areas of the minimal surfaces obtained by
Theorems A, B and \ref{th:isotopy} as in Remark~\ref{rem:estim}. For examples,
consider the figure eight knot complement $\boN$. We know there is a properly
embedded minimal once punctured torus $\Sigma$ in $\boN$ by
Theorem~\ref{th:isotopy} (Figure~\ref{fig:eight}). The Finite Total Curvature
Theorem \ref{th:ftc} and the Gauss equation tells us the area of $\Sigma$ is
strictly less than $2\pi$ (there are no embedded totally geodesic
surfaces in $\boN$).

What is the area of $\Sigma$ ? What is the properly embedded, non compact,
minimal surface of smallest area (it exists) in $\boN$ ? And in all such
manifolds $\boN$ ?

\bibliographystyle{plain}
\bibliography{../reference.bib}

\end{document}